\documentclass{article}

\usepackage[preprint]{neurips_2025}

\usepackage[utf8]{inputenc} % allow utf-8 input
\usepackage[T1]{fontenc}    % use 8-bit T1 fonts
\usepackage{hyperref}       % hyperlinks
\usepackage{url}            % simple URL typesetting
\usepackage{booktabs}       % professional-quality tables
\usepackage{amsfonts}       % blackboard math symbols
\usepackage{nicefrac}       % compact symbols for 1/2, etc.
\usepackage{microtype}      % microtypography
\usepackage{xcolor}         % colors
\usepackage{algorithm}
\usepackage{algpseudocode}
\usepackage{amsmath}
\usepackage{amssymb}
\usepackage{amsthm}
\usepackage{setspace}

\usepackage{subcaption}
\usepackage{graphicx} 

\newtheorem{theorem}{Theorem}  % Theorem environment
\newtheorem{lemma}{Lemma}
\newtheorem{definition}{Definition}

\title{An Alternating Approach to Approximate Dynamic Programming}

\author{%
  Di Zhang \\
  Department of Industrial and System Engineering\\
  University of Southern California\\
  Los Angeles, CA 90089 \\
  \texttt{dzhang22@usc.edu} \\
}

\begin{document}

\maketitle

\begin{abstract}
   This paper gives a new approximate dynamic programming (ADP) method to solve large-scale Markov decision programming (MDP) problem. In comparison with many classic ADP methods which have large number of constraints, we formulate an alternating ADP (AADP) which have both small number of constraints and small number of variables by approximating the decision variables (instead of the objective functions in classic ADP) and write the dual of the exact LP. Also, to get the basis functions, we use kernel approximation instead of empirical choice of basis functions, which can efficiently learn nonlinear functions while retaining the expressive power. By treating option pricing as an large-scale MDP problem, we apply the AADP method to give an empirical proof that American call option will not be exercised earlier if the underlying stock has no dividend payment, which is a classic result proved by Black-Scholes model. We also make comparison of pricing options in high-dimensional with some benchmark option pricing papers which use the classic ADP to give upper and lower bound of the option price.
\end{abstract}

\section{Introduction} Markov decision programming (MDP) are wildly used in different areas like health care~\cite{zhang2024state, zhang2024optimizing}, finance~\cite{lin2024economic}and optimal control~\cite{bertsekas1996dynamic}. There are many approaches to solve MDP problems, including but not limited to dynamic programming~\cite{bertsekas1996dynamic}, linear programming~\cite{de2003linear}, reinforcement learning~\cite{peng2024graph}, stochastic programming~\cite{zhang2024stochastic,zhang2024sampling}, and deep learning~\cite{dan2024evaluation,dan2024image}. However, the curse of dimensionality gives rise to prohibitive computational requirements that render infeasible the exact solution of large-scale
MDP~\cite{de2003linear}. One approach to deal with the problem is ADP, which try to approximate the objective function using a collection of basis functions. Some examples of this approach include De Farias and Van Roy~\cite{de2003linear}, Adelman and Mersereau \cite{adelman2008relaxations}, Desai et al.~\cite{desai2012pathwise} and Bertsimas and Misic~\cite{bertsimas2016decomposable} ,where they studied efficient methods based on linear programming for approximating solutions to such problems. These approach “fit” a linear combination of pre-selected basis functions to the dynamic programming cost-to-go function. However, although the number of variables is reduced, the number of constraints remains as large as in the exact LP. In this paper, we formulate an AADP which have both small number of constraints and small number of variables by approximating the decision variables (instead of the objective functions) and writing the dual of the exact LP. A detailed analysis is in section 2.

\noindent Another issue of ADP is the choice of basis functions. It mainly depends on the empirical results. For example, De Farias and Van Roy~\cite{de2003linear} uses polynomial functions, Desai, et al.~\cite{desai2012pathwise} uses $n + 2$ basis functions, and Mu Lin et al.~\cite{lin2024economic} employ the Hyperbolic Absolute Risk Aversion (HARA) Utility Function. Nevertheless, reasonably accurate knowledge of structure the cost-to-go function may be difficult in many problems. In this project, we will use kernel approximation to get the basis functions, which is highly related to Random Fourier Features: in Gaussian kernel approximation, replacing the random Gaussian matrix by a properly scaled random orthogonal matrix significantly decreases kernel approximation error~\cite{yu2016orthogonal}. With accurate kernel approximation, efficient nonlinear functions can be trained in the transformed space while retaining the expressive power of nonlinear methods~\cite{joachims2006training} . Further discussion will be shown in section 3 and some theoretical results are proved in section 4.   

\noindent Finally, a nice application of ADP is large-scale optimal stopping problem (since this problem can be formulated as MDP) such as option pricing. For example, it can be shown that American call option will not be exercised earlier if the underlying stock has no dividend payment. The classic proof is given by Black-Scholes model~\cite{black1973pricing}. In this report, we will using our AADP method to give the same empirical result. Also,  Desai et al.~\cite{desai2012pathwise} produce upper and lower bounds on the optimal value (the ‘price’) of a high-dimensional option pricing problem. We can use our method to make comparison with their results. The specific experimental comparison can be found in section 5.

\section{Problem Setup}
\noindent Let $S\times A$ be the state-action space and $\pi$ be a policy. Consider a constrained discounted MDP problem of the following kind:

\begin{equation} \label{MDP}
    \begin{aligned}
        & \max_{\pi} E[\sum_{t=0}^{\infty} \gamma^t r(x_t,u_t)],\\
        & s.t. \ E[\sum_{t=0}^{\infty} \gamma^t c(x_t,u_t)] \leq C
    \end{aligned}
\end{equation}

\noindent where $\gamma \in (0,1)$ is the discounted factor, $r(\cdot,\cdot): S\times A \to \mathbb{R}$ is the reward and $c(\cdot,\cdot): S\times A \to \mathbb{R}$ is the cost. It can be shown that (\ref{MDP}) can be rewritten as an LP:

\begin{equation} \label{DMDP}
    \begin{aligned}
        & \max_{\mu} \sum_{x,u} \mu(x,u) r(x,u)\\
        & s.t. \ \sum_{u} \mu(x,u) = \nu(x) + \gamma \sum_{x',u'}P_{x'x}(u')\mu(x',u') \quad \forall x,\\
        & \quad \ \ \sum_{x,u} c(x,u) \mu(x,u) \leq C\\
        & \quad \ \ \mu(x,u) \geq 0 \quad \forall x,u.
    \end{aligned}
\end{equation}

\noindent Note that (\ref{DMDP}) has $|S|\times|A|$ variables and $|S|\times|A|+|A|$ constraints, where $|S|$ is the size of state space and $|A|$ is the size of action space. Our goal is to reduce the number of variables and constraints so that problem (\ref{DMDP}) can be (approximately) solved efficiently.\\

\noindent To reduce the number of variables, let $\mu(x,u)=\sum_{i=1}^{k}\theta_i\phi_i(x,u) \geq 0$. Then (\ref{DMDP}) can be expressed as

\begin{equation} \label{AMDP}
    \begin{aligned}
        & \max_{\theta} \sum_{x,u} \sum_{i=1}^{k}\theta_i\phi_i(x,u) r(x,u)\\
        & s.t. \ \sum_{u} \sum_{i=1}^{k}\theta_i\phi_i(x,u) = \nu(x) + \gamma \sum_{x',u'}P_{x'x}(u')\sum_{i=1}^{k}\theta_i\phi_i(x',u') \quad \forall x.\\
        & \quad \ \ \sum_{x,u} c(x,u)\sum_{i=1}^{k}\theta_i\phi_i(x,u)  \leq C\\
        & \quad \ \ \theta_i \geq 0 \quad \forall i
    \end{aligned}
\end{equation}

\noindent Note that although we have reduced the number of variables to $k$, we still have $|S|$ constraints. Thus, to reduce the number of constraints, we will write the dual of problem $(\ref{AMDP})$ as 

\begin{equation} \label{DAMDP}
    \begin{aligned}
        & \min_{V,\omega} \sum_{x} V(x) \nu(x) + \omega \cdot C\\
        & s.t. \ -\sum_{x,u}\phi_i(x,u)r(x,u) + \sum_{x}V(x)\sum_{u}\phi_i(x,u) + \omega \sum_{x,u} c(x,u)\phi_i(x,u) \\
        & \geq \gamma \sum_{x,x',u'} V(x)P_{x'x}(u')\phi_i(x',u') \quad \forall i.
    \end{aligned}
\end{equation}

\noindent Let $V(x)=\sum_{j=1}^l \beta_j \psi_j(x)$. Then (\ref{DAMDP}) can be formulated as 

\begin{equation} \label{ADAMDP}
    \begin{aligned}
        & \min_{\beta} \sum_{x} \sum_{j=1}^l \beta_j \psi_j(x) \nu(x) + \omega \cdot (C+\varepsilon)\\
        & s.t. \ -\sum_{x,u}\phi_i(x,u)r(x,u) + \sum_{x,j} \beta_j \psi_j(x)\sum_{u}\phi_i(x,u) + \omega \sum_{x,u} c(x,u)\phi_i(x,u)\\
        & \geq \gamma \sum_{x,x',u',j}  \beta_j \psi_j(x)P_{x'x}(u')\phi_i(x',u') \quad \forall i.
    \end{aligned}
\end{equation}

\noindent Note that problem (\ref{ADAMDP}) has $l$ variables and $k$ constraints. Thus, we can approximately solve (\ref{DMDP}) using (\ref{ADAMDP}).  Solve linear Programming (\ref{ADAMDP}), we can get the dual variables $\theta_{1:k}$ and calculate 

$$\mu(x,u)=\sum_{i=1}^{k} \theta_i \phi_i(x,u). $$

\noindent Note that $\mu$ is a measure, and not necessarily a probability measure. To get a probability measure, we can easily normalize it as:

$$\Bar{\mu}(x,u)=\frac{\sum_{i=1}^{k} \theta_i \phi_i(x,u)} {\sum_{x,u} \sum_{i=1}^k \theta_i\phi_i(x,u)}. $$

\noindent Finally, the policy can be derived as:
$$\Bar{\pi}(u|x)=\frac{\Bar{\mu}(x,u)}{\sum_{u} \Bar{\mu}(x,u)}.$$
\hspace{2cm}

\section{Assumptions and Algorithm} \label{algorithm}

The analysis is based on the following assumptions: \\

A1: We can sample a finite S', which is a subset of S. \\

A2: There exists an oracle to get $P_{x',x}(u)$ for every $x,x' \in S'$ and $u \in U$. This assumption implies that the underlying Markov process are known for every $u \in U$. \\

A3: The matrices $P_{x',x}(u)$ are sparse. This assumption indicates that we can only transfer from state $x$ to only a few other states $x'$.\\

\subsection{Alternating ADP (AADP) Algorithm}

With the assumptions above, we can summarize the Alternating AMDP as follows.

\begin{algorithm}[H]
\setstretch{1.6}
\caption{AADP Algorithm}
\label{CGDM}
\begin{algorithmic}[1]

\State Given discount factor $\gamma$, approximation parameter $k$ and $l$, samples state space $S'$ and action space $U$.

\State Initialize $\nu(x)$ for every $x \in S'$.

\State Calculate $\phi_{1:k}(x,u)$ and $\psi_{1:l}(x)$ for every $x \in S'$ and $u \in U$.

\State Calculate $P_{x'x}(u)$ for every $x, x' \in S'$ and $u \in U$.

\State Solve linear Programming (\ref{ADAMDP}) and get the dual variables $\theta_{1:k}$.

\State Calculate $$\Bar{\mu}(x,u)=\frac{\sum_{i=1}^{k} \theta_i \phi_i(x,u)} {\sum_{x,u} \sum_{i=1}^k \theta_i\phi_i(x,u)} $$

\State Output policy $$\Bar{\pi}(u|x)=\frac{\Bar{\mu}(x,u)}{\sum_{u} \Bar{\mu}(x,u)}$$
%\State Output $\alpha^* \leftarrow \alpha^k$    
\end{algorithmic}
\end{algorithm}

\subsection{Kernelized AADP Algorithm}

In this section, we will present how to get basis functions $\phi_i$'s by kernel approximation based on Mercer's theorem and a theorem for shift-invariant kernel. Also, we will give a efficient approach to generate the approximated kernel mapping.

\begin{theorem}(Mercer's)
$K(x,y)=\sum_{j=1}^{\infty}\lambda_j\phi_j(x)\phi_j(y)$, where $\lambda_j \geq 0$, and $(\phi_j)_{j \geq 0}$ are orthonormal basis of $L_2$.
\end{theorem}

\noindent A kernel is called shift-invariant kernel if $K(x,y)=\psi(x-y)$. For example, RBF kernel is shift-invariant since $K_{RBF}(x,y)=e^{-\frac{||x-y||^2}{2\sigma^2}}$.

\begin{theorem}
Every shift-invariant kernel can be written as $K(x,y)=\int_{\mathbb{R}^n}p(w)e^{i\omega^T(x-y)}dw=E[e^{i\omega^T(x-y)}]$ for some density $P$ related to some probability distribution $F$ on $\mathbb{R}^n$.
\end{theorem}

\noindent Therefore, $K(x,y)$ can be approximated by $\hat{K}(x,y)=\frac{1}{m}\sum_{i=1}^m e^{i\omega_i^T(x-y)}$ with $\omega_i's$ draws i.i.d. from $F$. Then, we can rewrite $\hat{K}(x,y)=\frac{1}{m}\sum_{i=1}^m e^{i\omega_i^Tx}e^{-i\omega_i^Ty}=\frac{1}{m}e^{i\omega_i^Tx}(e^{i\omega_i^Ty})^*$, where $z^*$ is such a complex number that $zz^*=||z||^2$. Using Euler's formula, $e^{i\theta}=\cos{\theta}+i\sin{\theta}$, we could get:

\begin{equation}
\begin{aligned}
\hat{K}(x,y)& = \frac{1}{m}\sum_{i=1}^m(\cos{\omega_i^T x} \cos{\omega_i^Ty}+\sin{\omega_i^T x}\sin{\omega_i^Ty}) \\
& = \frac{1}{m}\sum_{i=1}^m[\cos{\omega_i^T x},\sin{\omega_i^T x}]^T[\cos{\omega_i^T y},\sin{\omega_i^T y}] \\
& = \hat{\phi}(x)^T\hat{\phi}(y),
\end{aligned}
\label{kapp}
\end{equation}

\noindent where $\hat{\phi}(x)=\frac{1}{\sqrt{m}}[\cos{\omega_i^T x},\sin{\omega_i^T x}]_{i=1}^{m} \in \mathbb{R}^{2m}$ is an approximate non-linear feature mapping (mapped to $2m$ dimensional space).

\noindent Note that for different kernels, $\omega$ comes from different distributions. Here, we will focus on RBF kernel, and we know that in this case $\omega$ follows Guassian distribution. Therefore, for RBF kernel, $\hat{\phi}(x)=\frac{1}{\sqrt{m}}[\cos{g_i^T x},\sin{g_i^T x}]_{i=1}^{m} \in \mathbb{R}^{2m}$.\\

\noindent We will then give an efficient way to get $G=[g_1,g_2,...,g_n]$, where $g_1,g_2,...,g_n$ are i.i.d. from Guassian distribution. Let 

\begin{equation} \label{HD}
\begin{aligned}
& H_0=[1]\\
& H_k=\begin{bmatrix}
H_{k-1} & H_{k-1}\\
H_{k-1} & -H_{k-1}
\end{bmatrix} \quad \forall k>0\\
\end{aligned}
\end{equation}

\noindent, and let $D=diag(d_1,d_2,...,d_n)$, where $d_i$'s are equally likely to be $-1$ and $+1$. One can verify that $HD$ is a good approximation of $G$ and every row of $HD$ is independent~\cite{Y2016}.

\noindent We summarize the ADDP algorithm using kernel approximation in algorithm $\ref{KAADP}$.

\begin{algorithm}[H]
\setstretch{1.6}
\caption{Kernelized AADP (KAADP) Algorithm}
\label{KAADP}
\begin{algorithmic}[1]

\State Given discount factor $\gamma$, approximation parameter $k$ and $l$, samples state space $S'$ and action space $U$.

\State Initialize $\nu(x)$ for every $x \in S'$.

\State Generate independent guassian random variables $g_{k,u}$ and $g_l$ for $k \in K$, $l \in L$ and $u \in U$ using $H$ in (\ref{HD}).

\State Calculate $\phi_{1:k}(x,u) = cos(g_{k,u}^T x)$ and $\psi_{1:l}(x) = cos(g_{l}^T x)$ for every $x \in S'$ and $u \in U$.

\State Calculate $P_{x'x}(u)$ for every $x, x' \in S'$ and $u \in U$.

\State Solve linear Programming (\ref{ADAMDP}) and get the dual variables $\theta_{1:k}$.

\State Calculate $$\Bar{\mu}(x,u)=\frac{\sum_{i=1}^{k} \theta_i \phi_i(x,u)} {\sum_{x,u} \sum_{i=1}^k \theta_i\phi_i(x,u)} $$

\State Output policy $$\Bar{\pi}(u|x)=\frac{\Bar{\mu}(x,u)}{\sum_{u} \Bar{\mu}(x,u)}$$
%\State Output $\alpha^* \leftarrow \alpha^k$    
\end{algorithmic}
\end{algorithm}

\section{Error bounds for ALP} \label{discussion}

In this subsection, we will talk about how good the approximation is. Formally, let $J^*$ be the optimal value of (\ref{DMDP}), and let $\hat{\theta}$ and $\hat{J}$ be the optimal solution and optimal value of (\ref{AMDP}), we want to somehow bound the difference between $J^*$ and $\hat{J}$ using some metrics.

\begin{definition}
The $(1,r)$ norm of matrix $A$ is defined by 

\begin{equation*}
    ||A||_{1,r}=|\sum_{x,u} A(x,u) r(x,u)|.
\end{equation*}

\end{definition}

\begin{lemma}
    Let $K$ be the feasible solution set which satisfies the constraints in problem (\ref{AMDP}). A vector $\hat{\theta}$ is an optimal solution to problem (\ref{AMDP}) if and only if it is an optimal solution to 
    
    \begin{equation} \label{d1}
        \min_{\theta \in K} ||\Phi \hat{\theta}(x,u)-\mu^*(x,u)||_{1,r}
    \end{equation}
\end{lemma}

\begin{proof} First, note that
    \begin{equation*}
        \begin{aligned}
         ||\Phi \hat{\theta}(x,u)-\mu^*(x,u)||_{1,r} & = |\sum_{x,u}(\Phi \hat{\theta}(x,u)-\mu^*(x,u))r(x,u)| \\ 
         & = \sum_{x,u} \mu^*(x,u) r(x,u)- \sum_{x,u}\Phi\hat{\theta}(x,u) r(x,u).
        \end{aligned}
    \end{equation*}
    
\noindent Thus, solving (\ref{d1}) is equivalent to solve 

\begin{equation*}
    \max_{\theta \in K} \sum_{x,u}\Phi\hat{\theta}(x,u) r(x,u),
\end{equation*}

\noindent which is exactly problem (\ref{AMDP}).
\end{proof}

\begin{lemma} \label{l2}
Suppose $r(x,u)=r(x)$, and $P_{x'x}(u')=P_{x'x}$. Let $\hat{J}(x)=\sum_{u} \Phi \hat{\theta} (x,u) \cdot r(x)$. Then we will have 

\begin{equation}
    \hat{J} = (I-\gamma \hat{P})^{-1} g
\end{equation}

\end{lemma}

\begin{proof}
    First, note that the constraints in problem (\ref{DMDP}) are 
    
    \begin{equation} \label{c3}
        \sum_{u} \Phi \hat{\theta} (x,u) = v(x)+\gamma \sum_{x',u'} P_{x'x} \Phi \hat{\theta} (x',u') \quad \forall x 
    \end{equation}

\noindent For each $x$, multiply \ref{c3} with $r(x)$, we will have 
    \begin{equation} \label{c3m}
        \sum_{u} \Phi \hat{\theta} (x,u) \cdot r(x) = v(x) \cdot r(x)+\gamma \sum_{x',u'} P_{x'x} \frac{r(x)}{r(x')} \Phi \hat{\theta} (x',u') r(x') \quad \forall x.
    \end{equation}
\noindent Note that this is equivalent to 

\begin{equation}
    \hat{J}(x)=v(x)r(x)+\gamma \sum_{x'}P_{x'x} \frac{r(x)}{r(x')} \hat{J}(x') \quad \forall x.
\end{equation}

\noindent Let $g(x)=v(x)r(x)$ and $\hat{P}_{x'x}=P_{x'x} \frac{r(x)}{r(x')}$, we obtain
    \begin{equation}
        \hat{J} = (I-\gamma \hat{P})^{-1} g
    \end{equation}
\end{proof}

\begin{theorem} \label{t4.3}

With the assumption in lemma \ref{l2}, we will have 

\begin{equation}
    ||J^*-\hat{J}||_{1} \leq \frac{1}{1-\gamma} ||\mu^*(x,u)-\Phi \hat{\theta}(x,u)||_{1,r} 
\end{equation}

\end{theorem}

\section{Implementation Details and  Experimental Result} \label{experiments}

\subsection{Option Pricing and the Optimal Stopping problem}

The well-known model for option pricing is BSM formula (we will use it as a baseline in our project). Let $N$ be the cumulative distribution function of the standard normal distribution, $T-t$ be the time to maturity, $S_t$ be the spot price of the underlying asset, $K$ be the strike (exercise) price $r$ be the risk free rate and $\sigma$ be the volatility of the underlying asset. The value $C(S_t,t)$ of a call option for a non-dividend-paying underlying stock in terms of the Black–Scholes parameters is:

\begin{equation} \label{bs}
    \begin{aligned}
        & C(S_t,t)=N(d_1)S_t-N(d_2)Ke^{-r(T-t)},\\
        & d_1=\frac{1}{\sigma\sqrt{T-t}}[\log(S_t/K)+(r+\frac{\sigma^2}{2})(T-t)],\\
        & d_2=d_1-\sigma\sqrt{T-t}\\
    \end{aligned}
\end{equation}

\noindent On the other hand, option pricing can also be treated as an optimal stopping problem in MDP. Thus, we can use the algorithm proposed in section 3 to do some experiment and compare with BSM method. Specifically, assume that the underlying stock follows geometric Brownian motion, i.e., the price process $\{S_t\}$ follows the stochastic differential equation

\begin{equation*}
    dS_t=\gamma S_t d_t + \sigma S_t d W_t, 
\end{equation*}

\noindent where $W_t$ is the standard Brownian motion. It can be shown that the probability density function of $S_t$ is

\begin{equation*} 
    f_{S_t}(s;\gamma,\sigma,t)=\frac{1}{\sqrt{2\pi}}\frac{1}{s \sigma \cdot t} \exp({-\frac{(\ln{\frac{s}{S_0}}-(\gamma-\frac{1}{2}\sigma^2)t)^2}{2 \sigma^2 t}}).
\end{equation*}

\noindent Using our setting in section 2, we can calculate $P_{x'x}(u)$ as

\begin{equation} \label{Prob}
    P_{x'x}(u)=\frac{1}{\sqrt{2\pi}}\frac{1}{s \sigma \cdot dt} \exp({-\frac{(\ln{\frac{x'}{x}}-(\gamma-\frac{1}{2}\sigma^2)dt)^2}{2 \sigma^2 dt}}).
\end{equation}

\noindent We summarize the main steps in algorithm \ref{KAADPO}.

\begin{algorithm}[H]
\setstretch{1.6}
\caption{KAADP for Option Pricing}
\label{KAADPO}
\begin{algorithmic}[1]

\State Given discount factor $\gamma$, approximation parameter $k$ and $l$, samples state space $S'$ and action space $U$.

\State Initialize $\nu(x)$ for every $x \in S'$.

\State Generate independent guassian random variables $g_{k,u}$ and $g_l$ for $k \in K$, $l \in L$ and $u \in U$ using $H$ in (\ref{HD}).

\State Calculate $\phi_{1:k}(x,u) = cos(g_{k,u}^T x)$ and $\psi_{1:l}(x) = cos(g_{l}^T x)$ for every $x \in S'$ and $u \in U$.

\State Calculate $P_{x'x}(u)$ using (\ref{Prob}) for every $x, x' \in S'$ and $u \in U$.

\State Solve linear Programming (\ref{ADAMDP}) and get the dual variables $\theta_{1:k}$.

\State Calculate $$\Bar{\mu}(x,u)=\frac{\sum_{i=1}^{k} \theta_i \phi_i(x,u)} {\sum_{x,u} \sum_{i=1}^k \theta_i\phi_i(x,u)} $$

\State Output policy $$\Bar{\pi}(u|x)=\frac{\Bar{\mu}(x,u)}{\sum_{u} \Bar{\mu}(x,u)}$$
%\State Output $\alpha^* \leftarrow \alpha^k$    
\end{algorithmic}
\end{algorithm}

\subsection{Experimental Result for American Call Option}

In this experiment, we set the initial stock price $S_0$, the risk free rate $\gamma=0.05$, the risk $\sigma=0.2$, the time period $T=1$ and $dt=1/100$, the exercise price $K_S=100$. Using BS formula, it can be shown that American call option will not be exercised earlier if the underlying stock has no dividend payment. Thus, the optimal policy is $\pi^*(u=1|x)=1$ for every $x \in S'$. Also, we can use BS formula to get the optimal pricing $P^*$ using (\ref{bs}). On the other hand, we can also use algorithm $\ref{KAADPO}$ to get a policy and use that policy to get an approximate price.

\noindent To make comparison with our alternating ADP method, we will first define $\varepsilon$-optimal policy. We say a policy $\Bar{\pi}$ is $\varepsilon$-optimal for $x \in S'$ if 

\begin{equation*}
    |\pi^*(u=1|x)-\Bar{\pi}(u=1|x)| < \varepsilon.
\end{equation*}

\noindent We will also define the rate of $\varepsilon$-optimal policy $\rho_{\varepsilon}$ as 

\begin{equation*}
    \rho_{\varepsilon} =\frac{|\{x \in S' \ | \ \textit{$\Bar{\pi}$ is $\varepsilon$-optimal for $x $} \}| }{|S'|}
\end{equation*}

\noindent Table \ref{tab1} gives the rate of $\varepsilon$-optimal policy for different number of basis functions $K$ and different sample size $I$. Note that the rate increase as we increase the number of basis functions or the sample size. From the table, we can recognize that given enough samples and basis functions, the policy $\bar{\pi}$ should be 

\begin{equation*}
    \bar{\pi}(u=1|x)=1 \quad \forall x \in S.
\end{equation*}

\noindent, which means we should keep holding the options until the end. This gives an experimental evidence that American call option will not be exercised earlier if the underlying
stock has no dividend payment.
% Table generated by Excel2LaTeX from sheet 'Sheet2'
\begin{table}[htbp]
\renewcommand{\arraystretch}{1.5}
\large
  \centering
  \caption{Rate of $\varepsilon$-optimal policy}
    \begin{tabular}{|c|c|c|c|}
    \hline
          & K=200   & K=300   & K=400 \\
    \hline
    I=200   & 0.36  & 0.7   & 1 \\
    \hline
    I=500   & 0.74  & 0.898 & 0.992 \\
    \hline
    I=1000  & 0.73  & 0.939 & 0.948 \\
    \hline
    \end{tabular}%
  \label{tab1}%
\end{table}%

\noindent We can also evaluate the behavior of the $\bar{\pi}$. Let $\Bar{P}$ be the approximate price we get after performing the policy $\bar{\pi}$. Table \ref{tab2} gives us the approximate price and the standard deviation after doing 1000 simulations.

\begin{table}[htbp]
\large
\renewcommand{\arraystretch}{1.5}
  \centering
  \caption{Approximate price for different initial prices $S_0$}
    \begin{tabular}{|c|c|c|c|c|c|}
    \hline
    $S_0$ & 80    & 90    & 100   & 110   & 120 \\
    \hline
    $P^*$ & 1.86  & 5.09  & 10.45 & 17.66 & 26.17 \\
    \hline
    $\bar{P}$ & 1.71  & 5.17  & 10.5  & 17.49 & 26.13 \\
    \hline
    Standard deviation & 0.18  & 0.35  & 0.5   & 0.62  & 0.75 \\
    \hline
    \end{tabular}%
  \label{tab2}%
\end{table}%

\subsection{Experimental Result for high-dimensional option pricing}

Similar to Desai et al.~\cite{desai2012pathwise}, we consider a Bermudan option over a calendar time horizon T defined on multiple assets (here we use 4 assets). The option has a total of M exercise opportunities at calendar times $\{dt, 2 \cdot dt, . . . , M \cdot dt\}$, where $dt=T/M$. The payoff of the option corresponds to that of a call option on the maximum of $n=4$ non-dividend paying assets with an up-and-out barrier.

\noindent The option is ‘knocked out’ (and worthless) at time $t$ if, at any of the times preceding and including $t$, the maximum of the n asset prices exceeded the barrier B. We let $y_t \in {0,1}$ serve as an indicator that the option is knocked out at time $t$. In particular, $y_t = 1$ if the option has been knocked out at time t or at some time prior, and $y_t = 0$ otherwise. The $\{y_t\}$ process evolves according to 

\begin{equation*}
    y_t = \bigg\{ \begin{aligned}
        & I \{\max_{j}p_0^j \geq B\} \quad if \ t=0;\\
        & min \{y_{t-1}, I \{\max_{j}p_t^j \geq B\}\} \quad if \ t>0.
    \end{aligned}
\end{equation*}

\noindent A state in the associated stopping problem is then given by the tuple $x=(p,y) \in \mathbb{R}^n \times \{0,1\}$, and
the payoff function $r(x,u)$ is defined according to

\begin{equation*}
    r(x,u=1)= (\max_{j} p_j(x)-K_S)^+(1-y(x)).
\end{equation*}

\noindent We will set strike price: K = 100; knock-out barrier price: B = 170; time horizon T = 3 years, risk-free rate: r = 5\% (annualized) and volatility $\sigma_j = 20\%$ (annualized). We also set the number of assets $n=4$ and the initial price $S_0=90,100,110$ common to all assets, and the assets are uncorrelated.

\noindent Table \ref{tab3} shows the upper and lower price bound for different method. Note that the result of our AADP method is not a neither a lower bound nor an upper bound. It is the average price after we do $N=100$ simulations.

\begin{table}[htbp]
\large
\renewcommand{\arraystretch}{1.5}
  \centering
  \caption{Upper and lower price bound for different method}
    \begin{tabular}{|r|c|c|c|c|c|c|c|}
    \hline
     $S_0$     & LS-LB & PO-LB & AADP  & Std   & DP-UB & PO-UB & DVF-UB \\
    \hline
    90    & 32.754 & 33.011 & 34.726 & 2.17  & 34.989 & 35.117 & 35.251 \\
    \hline
    100   & 40.797 & 41.541 & 43.332 & 1.78  & 43.587 & 43.853 & 44.017 \\
    \hline
    110   & 46.929 & 48.169 & 50.1  & 1.41  & 49.909 & 50.184 & 50.479 \\
    \hline
    \end{tabular}%
  \label{tab3}%
\end{table}%

To summarize, these experiments demonstrate the two primary merits to using the AADP method:

1. The greedy policy $\bar{\pi}$ which AADP method generate is close to the optimal policy $\pi^*$ if we generate enough samples and basis functions.

2. If we use $\bar{\pi}$ to simulate the price $\bar{P}$, then $\bar{P}$ is also close to the optimal price or in between the lower and upper bound for many benchmark methods.

\section{Limitations and Future Work}
\label{limitations}
\begin{itemize}
    \item \textbf{Generating sample states:} currently, we just generate samples uniformly. However, this is not very efficient because it is hard to recognize an appropriate sample size (In the experimental result, we know that only an appropriate sample size and number of basis functions will give us a``good'' policy). Casey, Michael S., and Suvrajeet Sen \cite{casey2005scenario} come up with an MSLP method which could find the best state to add (given that you have a sample state space $S'$, MSLP will tell you what is the best state to add next).

    \item \textbf{Theoretical proof}: we believe we can take advantage of some ideas in \cite{de2003linear} to proof theorem \ref{t4.3}.

    \begin{equation*}
    ||J^*-\hat{J}||_{1} \leq \frac{1}{1-\gamma} ||\mu^*(x,u)-\Phi \hat{\theta}(x,u)||_{1,r}. 
    \end{equation*}

    \item \textbf{High-dimensional option pricing:} currently, we only have $n=4$, we can also compare our AADP method with the benchmark upper and lower bound when $n=8$ and $n=16$.
\end{itemize}

\section{Conclusion}

A novel AADP method is introduced
to solve large-scale MDP problem. It is very powerful because it mainly include solving an LP with relatively small number of constraints and variables. We also gives some theoretical proofs of this method (although it is not complete). Also, instead of choosing basis functions empirically, we use kernel approximation (specifically, we use guassian kernel) to get the basis functions  which can
efficiently learn nonlinear functions while retaining the expressive power. Later, we apply our AADP method to price the American call option, which gives an experimental evidence that American call option will not be exercised earlier if the underlying
stock has no dividend payment (This is a classic result proved by Black-Scholes model). We also make comparisons with lower and upper bound in many benchmark methods. The results shows that the prices $\bar{P}$'s calculated by AADP are in
between those lower and upper bounds.     
\newpage

\bibliographystyle{abbrv}  % or another style of your choice
\bibliography{main}
%%%%%%%%%%%%%%%%%%%%%%%%%%%%%%%%%%%%%%%%%%%%%%%%%%%%%%%%%%%%

\appendix

\section{Technical Appendices and Supplementary Material} \label{appendix}

\subsection{An Alternative LP formulation of MDP}

In this subsection, we will show that problem (\ref{MDP}) is equivalent to problem (\ref{DMDP}). It will make the theory complete and give us an interpretation of $\mu(x,u)$.

\begin{theorem}
    Problem (\ref{MDP}) is equivalent to problem (\ref{DMDP}).
\end{theorem}

\begin{proof}
    Let us introduce an “occupation measure" $\mu(x, u)$, defined as
    \begin{equation*}
        \mu(x,u) = \sum_{k=0}^{\infty} \gamma^k P(x_k=x, u_k=u).
    \end{equation*}
    
    \noindent This can be interpreted intuitively as the discounted probability of occupancy in the State-Action space. Now, first note that the expectation in the objective (as well as the constraint) can be expressed as
    
    \begin{equation} \label{p1}
        \begin{aligned}
        E[\sum_{k=0}^{\infty} \gamma^k r(x_k,u_k)] & = \sum_{k=0}^{\infty} \sum_{x,u} \gamma^k r(x,u) P(x_k=x, u_k=u)\\
        & = \sum_{x,u} r(x,u) \mu(x,u).
        \end{aligned}
    \end{equation}
    
    \noindent Further, note that 
    
    \begin{equation} \label{p2}
        \begin{aligned}
            \sum_{u} \mu(x,u) & = P(x_0=x) + \gamma \sum_{k=0}^{\infty} \gamma^k P(x_{k+1}=x)\\
            & = \nu(x) + \gamma \sum_{k=0}^{\infty} \gamma^k \sum_{x'} \sum_{u'} P(x_{k+1}=x, x_k=x', u_k=u') \\
            & = \nu(x) + \gamma \sum_{x'} \sum_{u'} \sum_{k=0}^{\infty}   \gamma^k P(x_{k+1}=x, | x_k=x', u_k=u')\cdot P(x_k=x', u_k=u')\\
            & = \nu(x) + \gamma \sum_{x'} \sum_{u'} P(x, | x', u')\cdot \mu(x',u') \\
            & = \nu(x) + \gamma \sum_{x',u'}P_{x'x}(u')\mu(x',u').
        \end{aligned}
    \end{equation}
    
    \noindent (\ref{p1}) and (\ref{p2}) together imply that problem (\ref{MDP}) is equivalent to problem (\ref{DMDP}).
\end{proof}

\end{document}